\documentclass[12pt,leqno]{amsart}
\usepackage{amsmath,amsfonts,amssymb,amsthm,amscd,mathrsfs}
\usepackage{hyperref}
\theoremstyle{plain}
\newtheorem{theorem}{Theorem}[section]
\newtheorem{lemma}[theorem]{Lemma}

\theoremstyle{definition}
\newtheorem{defn}{Definition}[section]
\theoremstyle{remark}

\usepackage{mathtools}
\title[Quasi Yamabe solitons]{Some Characterizations of Quasi Yamabe solitons}
\author[A. A. Shaikh, and P. Mandal]{Absos Ali Shaikh$^{1*}$ and Prosenjit Mandal${^2}$}

\address{\noindent\newline $^{1,2}$Department of Mathematics,\newline The University of Burdwan,Golapbag,\newline Purba Bardhaman-713101,\newline West Bengal, India}
\email{$^1$aask2003@yahoo.co.in, aashaikh@math.buruniv.ac.in}
\email{$^2$prosenjitmandal235@gmail.com}
\begin{document}
\begin{abstract}
In this article, we have proved some results in connection with the potential vector field having finite global norm in quasi Yamabe soliton. We have derived some criteria in particular for the potential vector field on the non-positive Ricci curvature of the quasi Yamabe soliton. Also, a necessary condition for a compact quasi Yamabe soliton has been formulated. We further showed that if the potential vector field has a finite global norm in a complete non-trivial, non-compact quasi Yamabe soliton with finite volume, then the scalar curvature becomes constant and the soliton reduces to a Yamabe soliton.
\end{abstract}
\noindent\footnotetext{ $^*$ Corresponding author.\\
 $\mathbf{2010}$\hspace{5pt}Mathematics\; Subject\; Classification: 53C20; 53C21; 53C44. }
\maketitle

\section{Introduction and preliminaries}
An $n$-dimensional Riemannian manifold $(M,g)$ is said to be a quasi Yamabe soliton \cite{CD2018} if it possesses a smooth vector field $Y$ satisfying
\begin{equation}\label{y1}
2(R-\rho) g+2 \beta Y^{\flat}\otimes Y^{\flat}=\pounds_Yg,
\end{equation}
where $R$ denotes the scalar curvature of the metric $g$, $Y^{\flat}$ is the dual $1$-form of $Y$, $\pounds_Yg$ represents Lie derivative of the metric tensor $g$ in the direction of the vector field $Y$, $\rho$ is a constant and $\beta$ is a function on $M$. In the above equation ($\ref{y1}$) if we put $\beta=0$, then we get the following equation
\begin{equation*}
2(R-\rho) g=\pounds_Yg,
\end{equation*} 
and in this case $(M,g)$ is called a Yamabe soliton \cite{SCM21}. The Yamabe solitons are the special solutions of Yamabe flow, $\frac{\partial g}{\partial t}=-Rg$, which was introduced by Hamilton \cite{HA82, HA89}. If $\rho>0$, $\rho=0$ or $\rho<0$, then the Yamabe soliton is called shrinking, steady, or expanding, respectively.\\
In the local coordinates system the equation (\ref{y1}) can also be expressed in the following form
\begin{equation}\label{y2}
\nabla_iY_j+\nabla_jY_i=2(R-\rho) g_{ij} + 2 \beta Y_i Y_j.
\end{equation}
A Riemannian manifold $(M,g)$ is called a generalized quasi Yamabe gradient soliton \cite{NO16}, if $Y$ is gradient of some smooth function $f\in C^\infty(M)$ and in this case (\ref{y1}) takes the form 
\begin{equation}\label{e3}
\nabla_i\nabla_jf=(R-\rho)g_{ij} + \beta\nabla_if \nabla_jf.
\end{equation}
In the above equation $(\ref{e3})$ if we suppose $\beta$ as a constant, then in this case $(M,g)$ is called a quasi Yamabe gradient soliton \cite{HH14}.\\
\par In \cite{HH14}, Huang and Li proved that a compact quasi Yamabe gradient soliton is of constant scalar curvature. It has been shown in \cite{Wang13} that the scalar curvature of a non-compact quasi Yamabe gradient soliton may not be constant, and also showed that such a soliton has a warped product structure with warping function as the potential of the soliton. Later Neto
and Oliveira \cite{NO16}, have extended these results to the generalized quasi Yamabe
gradient solitons. By motivating of the above study and the results in \cite{AE14,CLY11,PW09,PRS11,ZH11,ZH09}, in the present paper, we have established that in a complete non-trivial non-compact quasi Yamabe soliton having finite volume, if the potential vector field has finite global norm, then $R$ is constant and the soliton reduces to a Yamabe soliton. Also, we have deduced certain conditions for the potential vector field on quasi Yamabe soliton of non-positive Ricci curvature and an integral condition for compact quasi Yamabe soliton. 
\par Let $A_m(M)$ be the space of all smooth $m$-forms in $M$, for any $m$, $0\leq m\leq n$. For $\eta, \zeta\in A_m(M)$, the local inner product $(\eta,\zeta)$ of $\eta$ and $\zeta$ is defined by \cite{MO01},
$$(\eta,\zeta)=\eta_{i_1,\cdots,i_m}\zeta^{i_1,\cdots,i_m},$$
where $\eta=\eta_{i_1,\cdots,i_m}du^{i_1}\wedge\cdots\wedge du^{i_m}$, $\zeta=\zeta_{i_1,\cdots,i_m}du^{i_1}\wedge\cdots\wedge du^{i_m}$ and $\zeta^{i_1,\cdots,i_m}=g^{i_1j_1}\cdots g^{i_mj_m}\zeta_{i_1,\cdots,i_m}$. For a fixed $m\geq 0$, the Hodge star operator $*:A_m(M)\rightarrow A_{n-m}(M)$ is defined by
$$*\eta=sgn(\mathscr{A},\mathscr{P})\eta_{i_1,\cdots,i_m}du^{j_1}\wedge\cdots\wedge du^{j_{n-m}},$$
for $\eta=\eta_{i_1,\cdots,i_m}du^{i_1}\wedge\cdots\wedge du^{i_m}\in A_m(M)$. Here $j_1<\cdots<j_{n-m}$ is the rearrangement of the complement of $i_1<\cdots <i_m$ in the set $\{1,...,n\}$ in increasing order and $sgn(\mathscr{A},\mathscr{P})$ is the sign of the permutation $i_1,...,i_m,j_1,...,j_{n-m}$. If $M$ is an oriented Riemannian manifold, then the global inner product $\langle \eta,\zeta\rangle$ of $\eta$ and $\zeta$ is defined by
$$\langle \eta,\zeta\rangle=\int_M \eta\wedge *\zeta,$$
and the global norm of $\eta$ is defined by $\|\eta\|^2=\langle \eta,\eta\rangle$ and remark that $\|\eta\|^2\leq \infty$. The natural adjoint operator of the exterior derivative $d:A_m(M)\rightarrow A_{m+1}(M)$ called the co-differential operator $\delta:A_m(M)\rightarrow A_{m-1}(M)$ is defined by
$$\delta=(-1)^m*^{-1}d*=(-1)^{n(m+1)+1}*d*.$$
For a $1$-form $\eta$, we obtain
\begin{eqnarray*}
(d\eta)_{ij}&=&\nabla_i\eta_j-\nabla_j\eta_i \quad \text{ and }\\
(\delta\eta)&=&-\nabla^i\eta_i,
\end{eqnarray*}
where $\nabla^i=g^{ij}\nabla_j$. For more details on co-differential operator and Hodge operator see \cite{MO01}. Let $A^0_m(M)$ is the subspace of $A_m(M)$ containing all $m$-forms in $M$ with compact support and the completion of $A^0_m(M)$ with respect to the global inner product is $L^2_m(M)$.
\par In this paper, the letter $Y$ denotes a vector field $Y=Y^i\partial_i$ and its dual $1$-form $Y=Y_idx^i=g_{ij}Y^jdx^i$, with respect to the Riemannian metric $g$.
\section{Main results}
\begin{defn}\cite{YO84}
 If $Y^{\flat}$ belongs to $L^2_1(M)\cap A_1(M)$, then $Y$ is said to be a vector field with finite global norm.
\end{defn}
 \par For a fixed point $z\in M$, $\tau(x)$ is the distance from $z$ to $x$, for each $x\in M$ and $\mathscr{B}_\tau$ is the open ball with center at $z$ and  radius $\tau>0$. Then for some constant $K>0$, the Lipschitz continuous function $\eta_\tau$ \cite{YA76}, satisfies
\begin{eqnarray*}
&& 0\leq \eta_\tau(x)\leq 1\quad\forall x\in M\\
&&|d\eta_\tau|\leq \frac{K}{\tau}\qquad \text{almost everywhere on }M\\
&&\eta_\tau(x)=1\quad\forall x\in \mathscr{B}_\tau\\
&& \text{ supp }\eta_\tau\subset \mathscr{B}_{2\tau}
\end{eqnarray*}
Then taking the limit as $\tau \rightarrow \infty$, we get $\lim\limits_{\tau\rightarrow\infty}\eta_\tau=1$.\\
 Now, we have
\begin{eqnarray}\label{l1}
  g(dY,dY)&=&
\nonumber \frac{1}{2}\{(\nabla_iY_j-\nabla_jY_i)(\nabla^iY^j-\nabla^jY^i) \}\\
\nonumber &=&\frac{1}{2}\{4(\nabla_iY_j)(\nabla^iY^j)-4(R-\rho)g_{ij}(\nabla^iY^j)-2\beta Y_iY_j(\nabla^iY^j+\nabla^jY^i)\}\\
\nonumber &=&\frac{1}{2}\{4(\nabla_iY_j)(\nabla^iY^j)-4(R-\rho)((R-\rho)n+\beta|Y|^2)-4\beta((R-\rho)|Y|^2 +\beta|Y|^4)\}\\
&=& 4g(\nabla Y,\nabla Y)-2n(R-\rho)^2-2s,
\end{eqnarray}
where $s=\{2\beta(R- \rho)+ \beta^2|Y|^2\}|Y|^2$. We choose $\beta$, $R$ and $\rho$ in such a way that $s\geq 0$.
Moreover,
\begin{eqnarray}\label{l2}
\nonumber g(\delta Y,\delta Y)&=&(\nabla^iY_i)(\nabla^jY_j)\\
\nonumber&=& g((R-\rho)n+ \beta |Y|^2,(R-\rho)n+ \beta |Y|^2)\\&=&|(R-\rho)n+ \beta |Y|^2|^2.
\end{eqnarray}
Then from the above two relations (\ref{l1}) and (\ref{l2}), we obtain the following lemma:
\begin{lemma}\label{L1}
In a quasi Yamabe soliton the potential vector field $Y$ satisfies the following relations:
\begin{equation}\label{l3}
4\|\eta_\tau\nabla Y\|^2_{\mathscr{B}_{2\tau}}-2n\|\eta_\tau(R-\rho)\|^2_{\mathscr{B}_{2\tau}}-2\|\eta_\tau\sqrt{s}\|^2=\|\eta_\tau dY\|^2_{\mathscr{B}_{2\tau}}
\end{equation}
and
\begin{equation}\label{l4}
\|\eta_\tau((R-\rho)n+\beta|Y|^2)\|^2_{\mathscr{B}_{2\tau}}=\|\eta_\tau\delta Y\|^2_{\mathscr{B}_{2\tau}}.
\end{equation}
\end{lemma}
Combining Lemma 2 and Lemma 3 of \cite{YO84}, we have
\begin{lemma}\cite{YO84}\label{L3}
For any smooth 1-form $Y$ in $M$, we have
\begin{equation}\label{l5}
4\langle \eta_\tau d\eta_\tau \otimes Y,\nabla Y\rangle_{\mathscr{B}_{2\tau}}+\langle \eta_\tau \nabla^2Y,\eta_\tau Y\rangle_{\mathscr{B}_{2\tau}}+2\langle \eta_\tau\nabla Y,\eta_\tau\nabla Y\rangle _{\mathscr{B}_{2\tau}}=0.
\end{equation}
\begin{eqnarray}
\nonumber &&\langle \eta_\tau\nabla^2 Y,\eta_\tau Y\rangle_{\mathscr{B}_{2\tau}}+\langle \eta_\tau dY,\eta_\tau dY\rangle_{\mathscr{B}_{2\tau}}+2\langle \eta_\tau dY,d\eta_\tau\wedge Y\rangle_{\mathscr{B}_{2\tau}}\\&&+\langle \eta_\tau\delta Y,\eta_\tau\delta Y\rangle_{\mathscr{B}_{2\tau}}-2\langle \eta_\tau\delta Y,*(d\eta_\tau\wedge*Y)\rangle_{\mathscr{B}_{2\tau}}
=\langle \eta_\tau \Re Y,\eta_\tau Y\rangle_{\mathscr{B}_{2\tau}},
\end{eqnarray}
where $(\nabla Y)_{ij}=\nabla_iY_j$, $(\nabla^2 Y)_i=\nabla^j\nabla_jY_i$ and the Ricci transformation on $A_1(M)$ is $(\Re Y)_i=R^j_iY_j$.
\end{lemma}
\begin{lemma}\cite{AN65}\label{L2}
For any smooth m-form $Y$ in $M$, there exists a $\tau$ independent positive constant $C$ satisfying
\begin{eqnarray*}
\frac{C}{\tau^2}\|Y\|^2_{\mathscr{B}_{2\tau}}&\geq&\|d\eta_\tau\otimes Y\|^2_{\mathscr{B}_{2\tau}},\\
\frac{C}{\tau^2}\|Y\|^2_{\mathscr{B}_{2\tau}}&\geq&\|d\eta_\tau\wedge Y\|^2_{\mathscr{B}_{2\tau}},\\
\frac{C}{\tau^2}\|Y\|^2_{\mathscr{B}_{2\tau}}&\geq&\|d\eta_\tau\wedge * Y\|^2_{\mathscr{B}_{2\tau}}.
\end{eqnarray*}
\end{lemma}
Using Lemma \ref{L1} and Lemma \ref{L2}, we have
\begin{eqnarray}
\nonumber |2\langle \eta_\tau dY,d\eta_\tau\wedge Y\rangle_{\mathscr{B}_{2\tau}}|&\leq & 2\|\eta_\tau dY\|_{\mathscr{B}_{2\tau}}\|d\eta_\tau\wedge Y\|_{\mathscr{B}_{2\tau}}\\
\nonumber&\leq & \frac{1}{4}\|\eta_\tau dY\|^2_{\mathscr{B}_{2\tau}}+4\|d\eta_\tau\wedge Y\|^2_{\mathscr{B}_{2\tau}}\\
\nonumber&\leq & \|\eta_\tau\nabla Y\|^2_{\mathscr{B}_{2\tau}}-\frac{n}{2}\|\eta_\tau(R-\rho)\|^2_{\mathscr{B}_{2\tau}}-\frac{1}{2}\|\eta_\tau \sqrt{s}\|^2_{\mathscr{B}_{2\tau}}\\
&&+\frac{4C}{\tau^2}\|Y\|_{\mathscr{B}_{2\tau}}^2.
\end{eqnarray}
\begin{eqnarray}
\nonumber|2\langle \eta_\tau \delta Y,*(d\eta_\tau\wedge*Y)\rangle_{\mathscr{B}_{2\tau}}|&\leq & 2\|\eta_\tau \delta Y\|_{\mathscr{B}_{2\tau}}\|d\eta_\tau\wedge *Y\|_{\mathscr{B}_{2\tau}}\\
\nonumber&\leq & \frac{1}{5}\|\eta_\tau \delta Y\|^2_{\mathscr{B}_{2\tau}}+5\|d\eta_\tau\wedge *Y\|^2_{\mathscr{B}_{2\tau}}\\
&\leq & \frac{1}{5}\|\eta_\tau((R-\rho)n+\beta|Y|^2)\|^2_{\mathscr{B}_{2\tau}}+\frac{5C}{\tau^2}\|Y\|_{\mathscr{B}_{2\tau}}^2.
\end{eqnarray}
Thus using Lemma \ref{L3}, we calculate:
\begin{eqnarray}\label{l6}
\nonumber\langle \eta_\tau\Re Y,\eta_\tau Y\rangle_{\mathscr{B}_{2\tau}}&=&-4\langle \eta_\tau d\eta_\tau\otimes Y,\nabla Y\rangle _{\mathscr{B}_{2\tau}}-2\langle \eta_\tau\nabla Y,\eta_\tau\nabla Y\rangle_{\mathscr{B}_{2\tau}}\\
\nonumber&&+\langle \eta_\tau dY,\eta_\tau dY\rangle_{\mathscr{B}_{2\tau}}+2\langle \eta_\tau dY,d\eta_\tau\wedge Y\rangle_{\mathscr{B}_{2\tau}}\\
\nonumber&&+\langle \eta_\tau\delta Y,\eta_\tau\delta Y\rangle_{\mathscr{B}_{2\tau}}-2\langle\eta_\tau\delta Y,*(d\eta_\tau\wedge *Y)\rangle_{\mathscr{B}_{2\tau}}\\
\nonumber&\geq & -\frac{1}{2}\|\eta_\tau\nabla Y\|^2 _{\mathscr{B}_{2\tau}}-\frac{8C}{\tau^2}\|Y\|^2 _{\mathscr{B}_{2\tau}}-2\|\eta_\tau\nabla Y\|^2 _{\mathscr{B}_{2\tau}}+4\|\eta_\tau\nabla Y\|^2 _{\mathscr{B}_{2\tau}}\\
\nonumber&&-2n\|\eta_\tau(R-\rho)\|^2 _{\mathscr{B}_{2\tau}}-2\|\eta_\tau \sqrt{s}\|^2 _{\mathscr{B}_{2\tau}}-\|\eta_\tau\nabla Y\|^2 _{\mathscr{B}_{2\tau}}\\
\nonumber&& +\frac{n}{2}\|\eta_\tau(R-\rho)\|^2_{\mathscr{B}_{2\tau}}+\frac{1}{2}\|\eta_\tau \sqrt{s}\|^2_{\mathscr{B}_{2\tau}}\\
&&-\frac{4C}{\tau^2}\|Y\|_{\mathscr{B}_{2\tau}}^2 +\|\eta_\tau((R-\rho)n+\beta|Y|^2)\|^2_{\mathscr{B}_{2\tau}}\\
\nonumber&&-\frac{1}{5}\|\eta_\tau((R-\rho)n+\beta|Y|^2)\|^2_{\mathscr{B}_{2\tau}}-\frac{5C}{\tau^2}\|Y\|_{\mathscr{B}_{2\tau}}^2\\
\nonumber&=& \frac{1}{2}\|\eta_\tau\nabla Y\|^2 _{\mathscr{B}_{2\tau}}-\frac{17C}{\tau^2}\|Y\|^2 _{\mathscr{B}_{2\tau}}-\frac{3n}{2}\|\eta_\tau(R-\rho)\|_{\mathscr{B}_{2\tau}}^2\\
\nonumber&&-\frac{3}{2}\|\eta_\tau\sqrt{s}\|_{\mathscr{B}_{2\tau}}+\frac{4}{5}\|\eta_\tau((R-\rho)n+\beta|Y|^2)\|^2 _{\mathscr{B}_{2\tau}}.
\end{eqnarray}
Again using Lemma \ref{L3}, we calculate:
\begin{eqnarray}\label{21}
\nonumber\langle \eta_\tau\Re Y,\eta_\tau Y\rangle_{\mathscr{B}_{2\tau}} \nonumber&\leq & \frac{1}{2}\|\eta_\tau\nabla Y\|^2 _{\mathscr{B}_{2\tau}}+\frac{8C}{\tau^2}\|Y\|^2 _{\mathscr{B}_{2\tau}}-2\|\eta_\tau\nabla Y\|^2 _{\mathscr{B}_{2\tau}}+4\|\eta_\tau\nabla Y\|^2 _{\mathscr{B}_{2\tau}}\\
\nonumber&&-2n\|\eta_\tau(R-\rho)\|^2 _{\mathscr{B}_{2\tau}}-2\|\eta_\tau \sqrt{s}\|^2 _{\mathscr{B}_{2\tau}}+\|\eta_\tau\nabla Y\|^2 _{\mathscr{B}_{2\tau}}\\
\nonumber&& -\frac{n}{2}\|\eta_\tau(R-\rho)\|^2_{\mathscr{B}_{2\tau}}-\frac{1}{2}\|\eta_\tau \sqrt{s}\|^2_{\mathscr{B}_{2\tau}}\\
&&+\frac{4C}{\tau^2}\|Y\|_{\mathscr{B}_{2\tau}}^2 +\|\eta_\tau((R-\rho)n+\beta|Y|^2)\|^2_{\mathscr{B}_{2\tau}}\\
\nonumber&&+\frac{1}{5}\|\eta_\tau((R-\rho)n+\beta|Y|^2)\|^2_{\mathscr{B}_{2\tau}}+\frac{5C}{\tau^2}\|Y\|_{\mathscr{B}_{2\tau}}^2\\
\nonumber&=& \frac{7}{2}\|\eta_\tau\nabla Y\|^2 _{\mathscr{B}_{2\tau}}+\frac{17C}{\tau^2}\|Y\|^2 _{\mathscr{B}_{2\tau}}-\frac{5n}{2}\|\eta_\tau(R-\rho)\|_{\mathscr{B}_{2\tau}}^2\\
\nonumber&&-\frac{5}{2}\|\eta_\tau\sqrt{s}\|_{\mathscr{B}_{2\tau}}+\frac{6}{5}\|\eta_\tau((R-\rho)n+\beta|Y|^2)\|^2 _{\mathscr{B}_{2\tau}}.
\end{eqnarray}
If $Y$ is a vector field of finite global norm, then (\ref{l6}) and (\ref{21}) reduces to the following inequalities
\begin{equation}\label{l7}
 \frac{1}{2}\|\nabla Y\|^2 -\frac{3n}{2}\|R-\lambda\|^2-\frac{3}{2}\|\sqrt{s}\|^2+ \frac{4}{5}\|(R-\rho)n+\beta|Y|^2\|^2\leq \limsup_{\tau\rightarrow\infty}\langle \eta_\tau\Re Y,\eta_\tau Y\rangle_{\mathscr{B}_{2\tau}},
\end{equation}
and 
\begin{equation}\label{22}
 \frac{7}{2}\|\nabla Y\|^2 -\frac{5n}{2}\|R-\rho\|^2-\frac{5}{2}\|\sqrt{s}\|^2+ \frac{6}{5}\|(R-\rho)n+\beta|Y|^2\|^2\geq \liminf_{\tau\rightarrow\infty}\langle \eta_\tau\Re Y,\eta_\tau Y\rangle_{\mathscr{B}_{2\tau}},
\end{equation}
respectively.
 Thus from $(\ref{22})$, we obtain the following theorem:
\begin{theorem}\label{th1}
Let $(M,g,Y,\beta)$ be a quasi Yamabe soliton. If $Y$ is a vector field of finite global norm, then the following inequality holds:
 \begin{equation*}
 \liminf_{\tau\rightarrow\infty}\langle \eta_\tau\Re Y,\eta_\tau Y\rangle_{\mathscr{B}_{2\tau}}\leq \frac{7}{2}\|\nabla Y\|^2+ \frac{6}{5}\|(R-\rho)n+\beta|Y|^2\|^2.
 \end{equation*}
\end{theorem}
\begin{theorem}
Let $(M,g,Y,\beta)$ be an $n(\geq 2)$-dimensional quasi Yamabe soliton with $\beta (R-\rho)\geq 0$ and Ricci curvature non-positive. If $Y$ is a vector field of finite global norm, then the following relation holds:
 \begin{equation*}
 \frac{5}{7}\|\nabla Y\|^2\leq \int_M {\beta^2|Y|^4}.
 \end{equation*}
\end{theorem}
\begin{proof}
If the Ricci curvature is non-positive, then
$$\limsup_{\tau\rightarrow\infty}\langle \eta_\tau\Re Y,\eta_\tau Y\rangle_{\mathscr{B}_{2\tau}}\leq 0.$$
Hence, (\ref{l7}) yields
\begin{equation*}
\frac{1}{2}\|\nabla Y\|^2 -\frac{3n}{2}\|R-\rho\|^2-\frac{3}{2}\|\sqrt{s}\|^2+ \frac{4}{5}\|(R-\rho)n+\beta|Y|^2\|^2\leq 0.
\end{equation*}
The above inequality together with the definition of global norm implies
\begin{eqnarray}
\nonumber 0&\geq&\frac{1}{2}\|\nabla Y\|^2-\frac{3n}{2}\int_{M} (R-\rho)^2-\frac{3}{2}\int_{M}\{2\beta(R- \rho)+ \beta^2|Y|^2\}|Y|^2+ \frac{4}{5}\int_{M}\{(R-\rho)n+\beta|Y|^2\}^2\\
\nonumber&\geq&\frac{1}{2}\|\nabla Y\|^2+\int_{M}\big\{-\frac{3n}{2} (R-\rho)^2-3\beta(R- \rho)|Y|^2-\frac{3}{2} \beta^2|Y|^4\\
\nonumber&+& \frac{4}{5}(R-\rho)^2n^2+\frac{8}{5}(R-\rho)n\beta|Y|^2+\frac{4}{5}\beta^2|Y|^4\big\}\\
\nonumber&\geq&\frac{1}{2}\|\nabla Y\|^2+\int_{M}\big\{(R-\rho)^2\big(\frac{4n^2}{5}-\frac{3n}{2}\big)+\beta(R- \rho)\big(\frac{8n}{5}-3\big)|Y|^2- \frac{3}{2}\beta^2|Y|^4+\frac{4}{5}\beta^2|Y|^4\big\}.
\end{eqnarray}
Using $n\geq 2$ and $\beta (R-\rho)\geq 0$ in the above inequality, we obtain
\begin{equation}
0\geq\frac{1}{2}\|\nabla Y\|^2-\frac{7}{10}\int_{M}\beta^2|Y|^4.
\end{equation}
This follows the desired result. 
\end{proof}

\begin{theorem}
Let $(M,g,Y,\beta)$ be a compact quasi Yamabe soliton, then
\begin{equation}
\int_{M}\{(n+2)|Y|^2 div(Y)+2(n-1)\beta|Y|^4\}=0.
\end{equation}
\end{theorem}
\begin{proof}
Taking trace of (\ref{y1}) we have
\begin{equation}\label{23}
div(Y)=(R-\rho)n+\beta |Y|^2.
\end{equation}
Now from (\ref{y1}) and (\ref{23}) we get
\begin{equation}
\frac{1}{2}(\pounds_Yg)(Y,Y)=\frac{1}{n}(div(Y)-\beta|Y|^2) g(Y,Y)+\beta Y_iY_j(Y,Y),
\end{equation}
which gives
\begin{equation}
ng(\nabla_YY,Y)=(div(Y)-\beta|Y|^2)|Y|^2+n\beta|Y|^4.
\end{equation}
Now
\begin{eqnarray}
\nonumber div(|Y|^2Y) &=& |Y|^2 div(Y)+2g(\nabla_Y Y,Y)\\
\nonumber & = & |Y|^2 div(Y)+\frac{2}{n} \{(div(Y)-\beta|Y|^2)|Y|^2+n\beta|Y|^4\}\\
&=&\frac{n+2}{n}|Y|^2 div(Y)+\frac{2(n-1)}{n}\beta|Y|^4,
\end{eqnarray}
which follows the proof of the theorem.
\end{proof}
\begin{theorem}
Let $(M,g,Y,\beta)$ be a completely non-compact non-trivial quasi Yamabe soliton whose volume is finite. If $Y$ is of finite global norm and one of the following conditions\\
(i) $\beta\geq 0$ and $R\geq \rho,$\\
(ii) $\beta\leq 0$ and $R\leq \rho,$\\
 holds, then the scalar curvature $R$ must be constant and $(M,g,Y,\beta)$ becomes a Yamabe soliton. 
\end{theorem}
\begin{proof}
For any positive $\tau$, we obtain
\begin{eqnarray}
\nonumber\frac{1}{\tau}\int_{\mathscr{B}_{2\tau}}|Y|dv & \leq & \Big(\int_{\mathscr{B}_{2\tau}}\langle Y,Y\rangle dv \Big)^{1/2} \Big(\int_{\mathscr{B}_{2\tau}}\Big(\frac{1}{\tau} \Big)^2 dv\Big)^{1/2}\\
\nonumber&\leq& \|Y\|_{\mathscr{B}_{2\tau}}\frac{1}{\tau}\Big(V(M)\Big)^{1/2},
\end{eqnarray}
where $V(M)$ represents the volume of $M$. Therefore, we get
\begin{equation*}
\liminf_{\ \tau\rightarrow \infty} \frac{1}{\tau}\int_{\mathscr{B}_{2\tau}}|Y|dv=0.
\end{equation*}
Again, we have
\begin{equation*}
\frac{C}{\tau}\int_{\mathscr{B}_{2\tau}}|Y|dv\geq\Big|\int_{\mathscr{B}_{2\tau}}\eta_\tau div Y dv \Big|.
\end{equation*}
Now with the help of quasi Yamabe soliton equation $(\ref{y1})$ we get
\begin{equation*}
\int_M \{(R-\rho)n+\beta |Y|^2\}dv=0,
\end{equation*}
The above relation together with $(i)$ or $(ii)$ gives $R=\rho$ and $\beta=0$. Which completes the proof.
\end{proof}

\section*{Acknowledgment}
 The second author gratefully acknowledges to the
 CSIR(File No.:09/025(0282)/2019-EMR-I), Govt. of India for the award of JRF.

\end{document}